\documentclass[a4paper,11pt]{article}
\usepackage[T1]{fontenc}
\usepackage{lmodern}
\usepackage{amsmath,amssymb,amsthm}
\usepackage{hyperref}
\usepackage[capitalise,nameinlink,noabbrev]{cleveref}
\crefname{assumption}{Assumption}{Assumptions}
\crefname{conjecture}{Conjecture}{Conjectures}
\crefname{lemma}{Lemma}{Lemma}
\crefname{proposition}{Proposition}{Propositions}
\crefname{equation}{}{}
\crefname{algocf}{Algorithm}{Algorithms}
\crefname{Question}{Question}{Questions}

\newtheorem{theorem}{Theorem}

\usepackage[a4paper, margin=3cm]{geometry}

\usepackage{graphicx} 
\usepackage{amsfonts,cite,verbatim,subcaption,tikz,float,multicol,url,booktabs,mathabx,mathrsfs}
\usetikzlibrary{positioning}
\usepackage[most]{tcolorbox}
\usepackage[section]{placeins}
\usepackage[algo2e,ruled,linesnumbered]{algorithm2e}

\newcommand{\argmax}{\mathop{\arg\max}}
\newcommand{\Ttran}{\mathsf{T}}

\newcommand{\fro}{\mathsf{F}}
\newcommand{\defi}{:=}

\newcommand{\normml}{\left\vert\kern-0.25ex\left\vert\kern-0.25ex\left\vert}  
\newcommand{\normmr}{\right\vert\kern-0.25ex\right\vert\kern-0.25ex\right\vert}
\newcommand{\order}{\mathcal{O}}

\newcommand{\R}{\mathbb{R}}

\newcommand{\mpr}{\mathbf{u}}
\newcommand{\fig}{eps}

\newcommand{\figsizeD}{0.45\textwidth}
\newcommand{\figsizeT}{0.3\textwidth}

\providecommand{\abs}[1]{\lvert#1\rvert}
\providecommand{\norm}[1]{\lVert#1\rVert}

\providecommand{\bigabs}[1]{\bigl\lvert#1\bigr\rvert}
\providecommand{\Bigabs}[1]{\Bigl\lvert#1\Bigr\rvert}

\usepackage[capitalise,nameinlink,noabbrev]{cleveref}

\usepackage{color}

\crefname{equation}{}{}
\crefname{Assumption}{Assumption}{Assumptions}
\crefname{algocf}{Algorithm}{Algorithms}
\title{Accelerating preconditioned Jacobi methods via perturbation-inspired pivoting
}
\author{
        Nian Shao\thanks{Institute of Mathematics, EPF Lausanne, 1015 Lausanne, Switzerland (\href{mailto:nian.shao@epfl.ch}{nian.shao@epfl.ch})}
        \and
        Yuji Nakatsukasa\thanks{Mathematical Institute, University of Oxford, Oxford, OX2 6GG, UK (\href{nakatsukasa@maths.ox.ac.uk}{nakatsukasa@maths.ox.ac.uk}).}
        }
\date{\today}

\begin{document}

\maketitle

\begin{abstract}
    Perturbation theory for symmetric matrices shows that eigenvalues with small spectral gaps are more sensitive to off-diagonal perturbation, implying that different  entries affect the eigenvalues unevenly. 
    Building on this insight, we incorporate spectral gap information into the Jacobi method for symmetric eigenvalue problems 
    and propose a new pivoting strategy, which is completely different from classical ones governed solely by the magnitude of the off-diagonal entries. 
    When combined with a mixed-precision preconditioner that diagonalizes the matrix to low precision, 
    numerical experiments demonstrate that the resulting strategy can significantly outperform the classical greedy approach when the original matrix has clustered eigenvalues.
\end{abstract}


\section{Introduction}
The Jacobi method, dating back to 1846 \cite{Jacobi1846}, is one of the oldest numerical methods for solving symmetric eigenvalue problems and singular value decomposition (SVD) \cite{Hestenes1958}. 
It applies a sequence of orthogonal transformations 
(Givens rotations in the simplest (non-block) case) 
to annihilate the off-diagonal entries of $A$ for eigenvalue problems (or $B^{\Ttran}B$ for SVD).
A key advantage of Jacobi's method is that it can achieve higher accuracy than methods based on tridiagonalization for eigenvalue problems (or bidiagonalization for SVD) 
for certain classes of matrices \cite{Demmel1992}. 
This high accuracy has recently become even more relevant as mixed-precision algorithms \cite{Higham2022} gain popularity with the growing support for low-precision arithmetic on modern hardware such as GPUs. Because the Jacobi method can be interpreted as a form of iterative refinement, it is particularly well suited to mixed-precision settings \cite{Gao2025,Higham2025,zhou2026computing} and is widely used for accurate eigenvalue and SVD computations.

The convergence of the Jacobi method depends strongly on the chosen pivoting strategy, i.e., the order in which the off-diagonal entries to be eliminated are chosen. The original greedy strategy proposed by Jacobi annihilates the largest off-diagonal entry at each step, yielding a linear convergence.
A recently proposed randomized pivoting strategy \cite{detherage2025unified} attains the same convergence rate in expectation. Another widely used strategy is the cyclic Jacobi method, which applies a full sweep over all index pairs $(i,j)$ with $i\neq j$, and repeats until convergence. As shown in \cite{Wilkinson1962}, the cyclic Jacobi method exhibits local quadratic convergence with respect to the number of sweeps. Because of its favorable local convergence properties, Jacobi's method is often combined with a preconditioning technique \cite{Veselic1989,Drmac2007,Drmac2007a} that transforms the matrix $A$ into an approximately diagonal form.

A key insight of this work is that the magnitude of the off-diagonal entries alone does not fully characterize their impact on the eigenvalue perturbation.  Eigenvalue perturbation bounds, see \cite{Li2005} for example, show that well-separated eigenvalues converge quadratically with respect to off-diagonal perturbations, meaning that eliminating all off-diagonal entries is unnecessary for accurate eigenvalue computation. In particular, a larger off-diagonal entry does not necessarily lead to a larger perturbation in the corresponding eigenvalues.
Consider the following example: 
\begin{equation*}
    A = \mathsf{blkdiag}\Biggl\{\begin{bmatrix}
        1 & 10^{-8} \\ 
        10^{-8} & 2
    \end{bmatrix}, \begin{bmatrix}
        1 & 10^{-9} \\ 
        10^{-9} & 1+10^{-10}
    \end{bmatrix} \Biggr\}.
\end{equation*}
With the classical greedy pivoting strategy, 
Jacobi's method will choose the pivot $(1,2)$ first because it has larger magnitude. However, in double precision, this entry is actually harmless for eigenvalues 
(that is, the diagonal entries already approximate the exact eigenvalues to within $10^{-16}$)
because the large spectral gap of 1 (between 1 and 2). In contrast, the smaller entry $10^{-9}$ introduces an error of $\order(10^{-9})$ to its eigenvalues due to the tiny gap of $10^{-10}$. This phenomenon indicates that we should take the spectral information into consideration when choosing the pivots.
In this paper, we propose a new greedy pivoting strategy (derived through a perturbation bound due to Li-Li, see~\eqref{eq:LiLi}):
\begin{equation}
    \label{eq:pivot}
    (i_{*},j_{*}) = \argmax_{i\neq j}L_{ij}(A),
    \quad\text{where}\quad
    L_{ij}(A)\defi  \frac{2\abs{a_{ij}}^{2}}{\abs{a_{ii}-a_{jj}}+\sqrt{\abs{a_{ii}-a_{jj}}^{2}+4\abs{a_{ij}}^{2}}}.
\end{equation}
Compared with the classical greedy strategy, our new formulation assigns different weights to different off-diagonal entries. This is inspired by the perturbation bound in \cite[Thm.~2]{Li2005}, which shows that eigenvalues with small spectral gaps are more sensitive to perturbations. For the example above, our strategy successfully chooses the pivot $(3,4)$.
In contrast to the classical strategy, our strategy may stop before transforming $A$ into a (numerically) diagonal matrix. Indeed, we show that once $L_{ij}(A)\leq \epsilon$ for all pivoting, then the diagonal entries of $A$ approximate the eigenvalues with error bounded by $2n(n-1)\epsilon$. 
When combined with a mixed-precision preconditioner, numerical results show that the resulting pivoting strategy can significantly outperform the classical greedy strategy, particularly when there are small gaps between the eigenvalues of the original matrix. 

\section{Jacobi's method and mixed-precision preconditioning}

Let $A\in\R^{n\times n}$ be a symmetric matrix with entries $a_{ij}$. Given a pivot $(i,j)$ with $i< j$, one step of Jacobi's method finds an orthogonal matrix $G_{ij}$ such that $G_{ij}^{\Ttran}A_{ij}G_{ij}$ is diagonal, where
\begin{equation*}
    G_{ij}=\begin{bmatrix}
    c &-s\\ 
    s & c
\end{bmatrix}
\quad\text{and}\quad
A_{ij}=\begin{bmatrix}
    a_{ii} & a_{ij}\\ 
    a_{ji} & a_{jj}
\end{bmatrix} 
\end{equation*}
Then the matrix is updated by an orthogonal similarity transformation with
\begin{equation*}
    G(c,s,i,j)\defi \begin{bmatrix}
    I_{i-1}&&&&\\ 
    &c&&-s&\\ 
    &&I_{j-i-1}&&\\ 
    &s&&c&\\ 
    &&&&I_{n-j}
\end{bmatrix}.
\end{equation*} 
After a sequence of transformations, the Jacobi's method aims to annihilate all off-diagonal entries and transform $A$ into a diagonal matrix whose diagonal entries are its eigenvalues.

For the computation of SVD of $B=[b_{1},\dotsc,b_{n}]\in\R^{m\times n}$ with $m\geq n$, the one-sided Jacobi essentially applies Jacobi's method for $A=B^{\Ttran}B$. Specifically, given a pivot $(i,j)$, we find an orthogonal matrix $G_{ij}$ such that 
the two columns $[b_{i},b_{j}]G_{ij}$ become orthogonal, and then update $B$ by multiplying $G(c,s,i,j)$ from the right. Unlike in eigenvalue computation, the one-sided Jacobi method aims to make the columns of $B$ orthogonal (not necessarily orthonormal) and then extracts the singular values from the norms of these columns.
The pseudocode for Jacobi's method for computing eigenvalues and SVD are summarized in \cref{algo:EVP,algo:SVD}, respectively.

\begin{algorithm2e}[htbp]
    \caption{Jacobi method for eigenvalue computation}
    \label{algo:EVP}
    \KwIn{A symmetric matrix $A$\;}
    \While{Not convergent}{
        Find a pivot $(i,j)$\;
        Compute $c$ and $s$ such that $G_{ij}^{\Ttran}A_{ij}G_{ij}$ is diagonal\;
        Update $A$ by $A\gets G(c,s,i,j)^{\Ttran}AG(c,s,i,j)$\;   
    }
\end{algorithm2e}

\begin{algorithm2e}[htbp]
    \caption{One-sided Jacobi method for SVD}
    \label{algo:SVD}
    \KwIn{A matrix $B\in\R^{m\times n}$ with $m\geq n$\;}
    \While{Not convergent}{
        Find a pivot $(i,j)$\;
        Compute $c$ and $s$ such that $[b_{i},b_{j}]G_{ij}$ contains orthogonal columns\;
        Update $B$ by $B\gets BG(c,s,i,j)$\;   
    }
\end{algorithm2e}

 \subsection{Mixed-precision preconditioning} 
 
As mentioned above, preconditioning can significantly accelerate the convergence of Jacobi's method. Here we describe the mixed-precision preconditioner.
Denote the working precision and a lower precision by $\mpr$ and $\mpr_{\ell}$, respectively, where $\mpr\ll \mpr_{\ell}$. The mixed-precision preconditioned Jacobi method for eigenvalue problems first computes a spectral decomposition in lower precision:
\begin{equation*}
    A = V_{\ell}\Lambda_{\ell}V_{\ell}^{\Ttran}+\order(\norm{A}\mpr_{\ell}),
\end{equation*}
where $V_{\ell}$ is orthogonal in lower precision, that is $V_{\ell}^{\Ttran}V_{\ell}-I=\order(\mpr_{\ell})$. Then we reorthogonalize $V_{\ell}$ to working precision, denoted by $V$, by the Newton--Schulz iteration or Householder QR decomposition. The matrix $V$ is called a preconditioner to $A$, and we apply Jacobi's method to the preconditioned matrix $V^{\Ttran}AV$, whose off-diagonal entries are $\order(\norm{A}\mpr_{\ell})$.

For the SVD, we first compute the right singular vectors $V_{\ell}$ in lower precision, obtaining
\begin{equation*}
    B=U_{\ell}\Sigma_{\ell}V_{\ell}^{\Ttran}+\order(\norm{B}\mpr_{\ell}),
\end{equation*}
and reorthogonalize it in working precision to obtain $V$. We next apply the one-sided Jacobi method to the preconditioned matrix $BV$, whose columns are orthogonal up to $\order(\norm{B}\mpr_{\ell})$. 
To use our pivoting strategy for the SVD, one can take $(BV)^{\Ttran}BV$ as the matrix $A$ after preconditioning in what follows.

\section{A new pivoting strategy based on eigenvalue perturbation}
Given two symmetric matrices of the form
\begin{equation*}
    T = \begin{bmatrix}
        T_{1} &\\ 
        & T_{2}
    \end{bmatrix}
    \quad\text{and}\quad 
    \widehat{T} = \begin{bmatrix}
        T_{1} & E^{\Ttran} \\ 
        E & T_{2}
    \end{bmatrix},
\end{equation*}
Li and Li \cite[Thm.~2]{Li2005} established a powerful eigenvalue perturbation bound as 
\begin{equation}
    \label{eq:LiLi}
    \Bigabs{\lambda_{i}(T)-\lambda_{i}(\widehat{T})} \leq \frac{2\norm{E}^{2}}{\eta_{i}+\sqrt{\eta_{i}^{2}+4\norm{E}^{2}}},
    \quad\text{where}\quad 
    \eta_{i}\defi \max_{j=1,2}\min_{\mu\in\Lambda(T_{j})} \abs{\lambda_{i}(T)-\mu}.
\end{equation}
The Li--Li bound \cref{eq:LiLi} shows that when the spectral gap $\eta_{i}$ is large, the eigenvalue perturbation is quadratic in $\norm{E}$, whereas a small gap results in a larger, linear perturbation, as described by Weyl's inequality.

In Jacobi's method for eigenvalue computation, we approximate eigenvalues of $A$ by its diagonal entries. Therefore, we want the off-diagonal entries to perturb the eigenvalues as little as possible. If we take 
\begin{equation*}
    \widehat{T}=A_{ij} = \begin{bmatrix}
        a_{ii} & a_{ij}\\ 
        a_{ji} & a_{jj}
    \end{bmatrix}, 
\end{equation*}
then the function $L_{ij}(A)$ defined in \cref{eq:pivot} is precisely the Li--Li's bound \cref{eq:LiLi}, which quantifies how much the off-diagonal entry $a_{ij}$ perturbs the eigenvalues of $A_{ij}$. Our new strategy \cref{eq:pivot} therefore follows a greedy approach that attempts to annihilate the off-diagonal entry causing the largest potential eigenvalue perturbation. This argument is informal as it is based on the $2\times 2$ matrix $A_{ij}$ and not the whole matrix $A$. We next show that this strategy nonetheless allows us to give strong guarantees on the eigenvalue accuracy globally.

\subsection{Convergence analysis}

In general, our strategy does not guarantee that $A$ will converge to a diagonal matrix, because $L_{ij}(A)$ can be of $\order(\mpr)$ even when $\abs{a_{ij}} \gg \mpr$. Instead, the following theorem shows that if $L_{ij}(A) \leq \epsilon$ for all $i \neq j$, then the diagonal entries of $A$ indeed approximate its eigenvalues with accuracy $2n(n-1)\epsilon$.
Thus, for any prescribed tolerance $\tau$, 
one can guarantee eigenvalue accuracy of $\tau$ 
by setting $\epsilon=\tau/(2n(n-1))$ and running Jacobi's method until $L_{ij}(A) \leq \epsilon$.

\begin{theorem}
    \label{thm:conv}
    Let $A\in\R^{n\times n}$ be a symmetric matrix with $(i,j)$ entry $a_{ij}$, where $a_{11}\geq a_{22}\geq  \dotsb \geq a_{nn}$. Suppose that for all $i\neq j$,
    \begin{equation*}
        L_{ij}(A)=\frac{2\abs{a_{ij}}^{2}}{\abs{a_{ii}-a_{jj}}+\sqrt{\abs{a_{ii}-a_{jj}}^{2}+4\abs{a_{ij}}^{2}}}\leq \epsilon.
    \end{equation*} 
    Denoting the $k$-th largest eigenvalue of $A$ by $\lambda_{k}(A)$, we have   
    \begin{equation*}
        \bigabs{\lambda_{k}(A)-a_{kk}}\leq 2n(n-1)\epsilon.
    \end{equation*}  
\end{theorem}
We note that the diagonals of $A$ are assumed to be ordered for convenience; if not, one can apply the theorem to a permuted matrix $P^TAP$ so that the assumption holds. 

\begin{proof}
    The proof is based on mathematical induction for the matrix size $n$. The situation for $n=2$ is a direct result of the Li--Li's bound \cref{eq:LiLi}. 
    
    For a general $n>2$, we consider $\lambda_{k}(A)$ and $a_{kk}$. 
    From $L_{ij}(A)\leq \epsilon$, we know that 
    \begin{equation*}
        2\abs{a_{ij}}^{2}-\epsilon\abs{a_{ii}-a_{jj}}\leq  \epsilon\sqrt{\abs{a_{ii}-a_{jj}}^{2}+4\abs{a_{ij}}^{2}}.
    \end{equation*}
    implying that $\abs{a_{ij}}^{2}\leq \epsilon \abs{a_{ii}-a_{jj}}$ or
    \begin{equation*}
        \begin{aligned}
            &0<2\abs{a_{ij}}^{2}-\abs{a_{ii}-a_{jj}}\leq \epsilon\sqrt{\abs{a_{ii}-a_{jj}}^{2}+4\abs{a_{ij}}^{2}}\\ 
        &\Longleftrightarrow    
        4\abs{a_{ij}}^{4}-4\epsilon\abs{a_{ii}-a_{jj}}\abs{a_{ij}}^{2}+\epsilon^{2}\abs{a_{ii}-a_{jj}}^{2}\leq \epsilon^{2} \abs{a_{ii}-a_{jj}}^{2}+4\epsilon^{2}\abs{a_{ij}}^{2}\\ 
        &\Longleftrightarrow    
        \abs{a_{ij}}^{2}\leq \epsilon\abs{a_{ii}-a_{jj}}+ \epsilon^{2}.
        \end{aligned} 
    \end{equation*}
    Consequently, letting $g=a_{11}-a_{nn}$, we obtain
    \begin{equation*}
        \abs{a_{ij}}^{2}\leq \epsilon \abs{a_{ii}-a_{jj}}+\epsilon^{2} \leq  \epsilon (g+\epsilon), 
    \end{equation*}
    according to $L_{ij}(A)\leq \epsilon$. Moreover, recalling that $a_{11}\geq \dotsb \geq a_{nn}$, we have 
    \begin{equation*}
        \begin{aligned}
            \norm{A-D}_{\fro}^{2} &= \sum_{i\neq j}\abs{a_{ij}}^{2} \leq 2\epsilon \Bigl(\sum_{i<j}(a_{ii}-a_{jj})\Bigr)+n(n-1)\epsilon^{2}\\ 
            & =2\epsilon \Bigl(\sum_{i=1}^{\lfloor n/2 \rfloor}(n-2i+1)(a_{ii}-a_{n-i+1,n-i+1})\Bigr)+n(n-1)\epsilon^{2} \leq \frac{n^{2}}{2}\epsilon g+n(n-1)\epsilon^{2}.
        \end{aligned}
    \end{equation*}

    We consider two cases. First if $g\leq 8n(n-2)\epsilon$, by Weyl's inequality, we know that 
    \begin{equation*}
        \bigabs{\lambda_{k}(A)-a_{kk}}^{2}\leq \norm{A-D}^{2}\leq \norm{A-D}_{\fro}^{2} \leq  \bigl(4n^{3}(n-2)+n(n-1)\bigr)\epsilon^{2}\leq \bigl(2n(n-1)\epsilon\bigr)^{2}.
    \end{equation*}
    
    If $g>8n(n-2)\epsilon$, without loss of generality, we assume that $a_{11}-a_{kk}\leq a_{kk}-a_{nn}$, implying that $k\neq n$ and $a_{kk}-a_{nn}\geq g/2$.
    Partition the matrix $A$ as 
    \begin{equation*}
        A = \begin{bmatrix}
            \widehat{A} & E^{\Ttran}\\ 
            E & a_{nn}
        \end{bmatrix}.
    \end{equation*}
    Note that $\widehat{A}$ is an $(n-1)\times (n-1)$ symmetric matrix satisfying $L_{ij}(\widehat{A})\leq \epsilon$ for all $i\neq j$.
    By the induction assumption, we know that 
    \begin{equation*}
        \bigabs{\lambda_{k}(\widehat{A})-a_{kk}}\leq 2(n-1)(n-2)\epsilon.
    \end{equation*}
    Applying Li--Li's bound \cref{eq:LiLi} to $A$, we obtain 
    \begin{equation*}
        \bigabs{\lambda_{k}(\widehat{A})-\lambda_{k}(A)}\leq \frac{2\norm{E}^{2}}{\lambda_{k}(\widehat{A})-a_{nn}+\sqrt{\abs{\lambda_{k}(\widehat{A})-a_{nn}}^{2}+4\norm{E}^{2}}}
        \leq  \frac{4\norm{E}^{2}}{\widehat{g}+\sqrt{\widehat{g}^{2}+16\norm{E}^{2}}}, 
    \end{equation*}
    where we use 
    \begin{equation*}
        \lambda_{k}(\widehat{A})-a_{nn} \geq a_{kk}-a_{nn} -\bigabs{\lambda_{k}(\widehat{A})-a_{kk}} \geq \frac{g}{2}-2(n-1)(n-2)\epsilon =: \frac{\widehat{g}}{2}.
    \end{equation*}
    From $L_{nj}(A)\leq \epsilon$, we know that 
    \begin{equation*}
        \norm{E}^{2}=\sum_{j=1}^{n-1}\abs{a_{nj}}^{2} \leq (n-1)\epsilon (g+\epsilon).
    \end{equation*}
    Thus, plugging it into the Li--Li's bound and using 
    \begin{equation*}
        \widehat{g} = g-4(n-1)(n-2)\epsilon\geq \frac{n+1}{2n}g>\frac{g}{2},
    \end{equation*}
    we know that 
    \begin{equation*}
        \bigabs{\lambda_{k}(\widehat{A})-\lambda_{k}(A)} \leq \frac{4(n-1)\epsilon (g+\epsilon)}{\widehat{g}+\sqrt{\widehat{g}^{2}+16(n-1)(\epsilon g+\epsilon^{2})}}\leq 4(n-1)\epsilon.
    \end{equation*}
    The proof is finished by the triangular inequality:
    \begin{equation*}
        \bigabs{\lambda_{k}(A)-a_{kk}}\leq 
        \bigabs{\lambda_{k}(A)-\lambda_{k}(\widehat{A})}+
        \bigabs{\lambda_{k}(\widehat{A})-a_{kk}}\leq 2n(n-1)\epsilon.
    \end{equation*} 
\end{proof}

\subsection{Relative complexity}
Compared with the classical greedy pivoting strategy based on $\abs{a_{ij}}$, evaluating the new pivoting criterion $L_{ij}$ in~\cref{eq:pivot} is slightly more expensive. In the initial phase, each $L_{ij}(A)$ can be computed using $\order(1)$ operations, for a total of $\order(n^2)$ operations. After a Jacobi rotation is applied, the values $L_{ij}(A)$ must be updated. Since only the corresponding $i$th and $j$th rows and columns are affected, only $\order(n)$ such quantities require updating. Therefore, relative to the greedy strategy, the additional computational overhead remains within a constant factor.

\section{Numerical experiments}

To illustrate the effectiveness of our new pivoting strategy, we perform the following numerical experiment\footnote{All numerical experiments were implemented in MATLAB~R2022b and carried out on an AMD Ryzen~9 6900HX processor (8 cores, 3.3--4.9GHz) with 32GB of RAM. Scripts for reproducing the numerical results are publicly available at \url{https://github.com/nShao678/Jacobi}
} with a preconditioned Jacobi method for eigenvalue computation, where the mixed-precision preconditioner is applied. 
Here, we use the Householder QR decomposition to reorthogonalize the eigenvector from lower precision. The lower precision and working precision are \texttt{single} and \texttt{double} in Matlab, respectively, that is, $\mpr_{\ell}\approx 10^{-8}$ and $\mpr\approx 10^{-16}$. The error of eigenvalues is measured by the total relative error $\sum_{i=1}^{n}\abs{\lambda_{i}-\widehat{\lambda}_{i}}/\norm{A}$,
where $\lambda_{i}$ and $\widehat{\lambda}_{i}$ denote the exact and computed eigenvalues of $A$, respectively.

\subsection{Artificial examples with clustered eigenvalues}

We consider $A =V\Lambda V^{\Ttran}$,
where $V\in\R^{1000\times 1000}$ is an orthogonal matrix and $\Lambda = \mathsf{blkdiag}\{\Lambda_{1},\Lambda_{2}\}$.
For $i=1,2$, the matrices $\Lambda_{i}$ consist of $\lambda_{i}^{(k)}$ with     
\begin{equation*}
    \lambda_{1}^{(k)} = 1+\exp(-\alpha k)
    \quad\text{and}\quad 
    \lambda_{2}^{(k)} = 2+\frac{k-1}{499}
    \quad\text{for}\quad k=1,\dotsc,500,
\end{equation*}
where $\alpha>0$ is a parameter to control the tightness of the eigenvalue cluster. Note that $\Lambda_{1}$ forms an eigenvalue cluster around $1$, while $\Lambda_{2}$ consists of eigenvalues that are equally spaced in the interval $[2,3]$ and are well separated from $\Lambda_{1}$.

Using $\alpha \in \{0.01,1\}$ and applying both the classical greedy pivoting strategy and our new strategy in \cref{eq:pivot}, we report the numerical results in \cref{fig:clu}. Our new pivoting strategy outperforms the classical greedy strategy slightly for $\alpha=0.01$, and significantly for $\alpha=1$, which correspond to the situation with tight eigenvalue cluster. Moreover, we remark that convergence of the eigenvalues does not necessarily imply convergence of the off-diagonal entries. This is because, according to the Li--Li's bound, when the eigenvalues are well separated, even relatively large off-diagonal entries have little effect on the eigenvalue accuracy. Indeed, due to the large gap between $\Lambda_{1}$ and $\Lambda_{2}$, the convergence of eigenvalues is much faster than that of the off-diagonal entries in our strategy, but not in the classical greedy one.

\begin{figure}[htbp]
    \centering
    \includegraphics[width=\figsizeD]{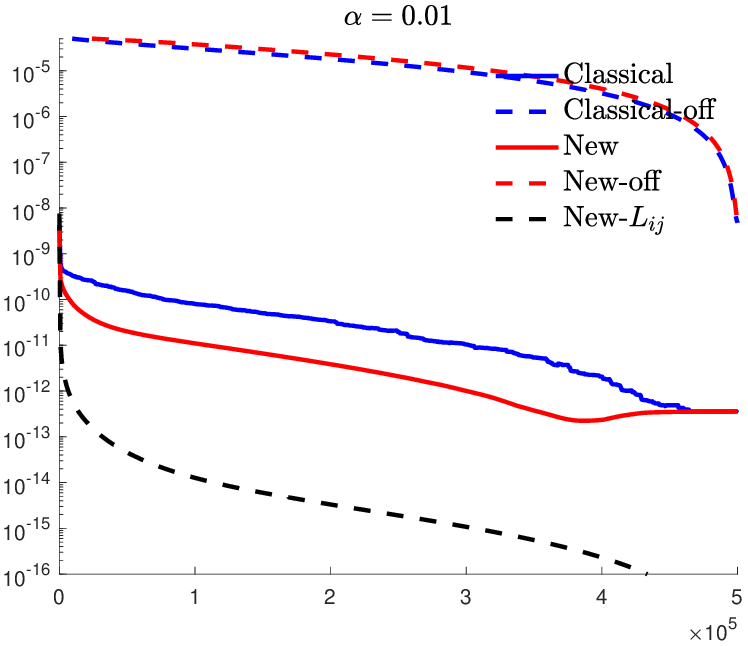}
    \includegraphics[width=\figsizeD]{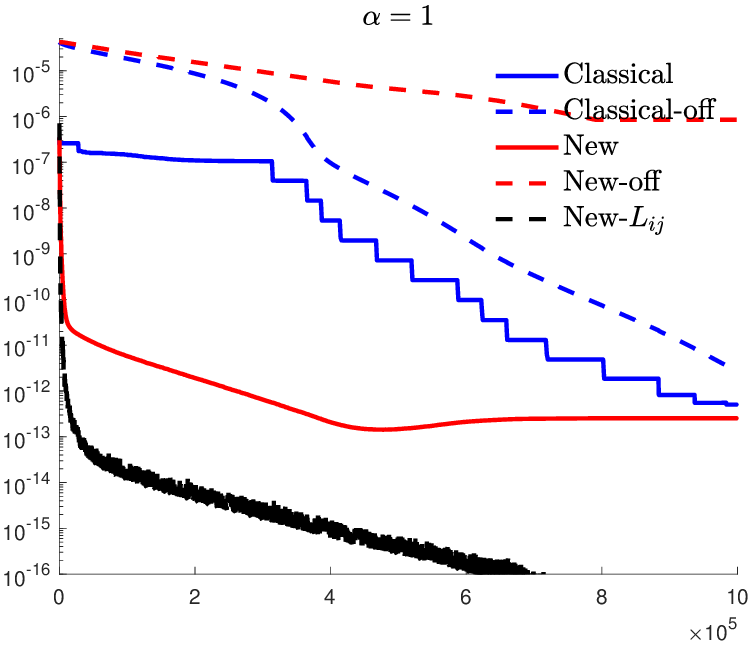}
    \caption{Total relative error of eigenvalues (solid) and Frobenius norm of off-diagonal entries (dashed) \textit{vs.} number of iterations. The black dashed lines are the largest $L_{ij}(A)$ at each step.}
    \label{fig:clu}
\end{figure}

To further investigate the difference in convergence, we look into the convergence history when $\alpha=1$. We first permute  the matrix after preconditioning such that the eigenvalue cluster is in the top-left corner, that is, 
\begin{equation*}
    \widetilde{A} = \begin{bmatrix}
        \widetilde{A}_{11} & \widetilde{A}_{12}\\ 
        \widetilde{A}_{21} & \widetilde{A}_{22}
    \end{bmatrix} = \begin{bmatrix}
        \Lambda_{1} &\\ 
        & \Lambda_{2}
    \end{bmatrix}+\order(\mpr_{\ell}).
\end{equation*}
Then we plot the magnitudes of the matrix entries after 100000, 200000, and 300000 iterations for both the classical and the new pivoting strategies in \cref{fig:hist}. We see that the two strategies eliminate off-diagonal entries in completely different ways. The classical strategy first annihilates the off-diagonal entries in $\widetilde{A}_{22}$ and then those in $\widetilde{A}_{12}$ and $\widetilde{A}_{21}$, whereas our new strategy begins with $\widetilde{A}_{11}$ and then proceeds to $\widetilde{A}_{22}$. 
Recall that $\widetilde{A}_{11}$ contains a cluster of eigenvalues. 
The nonzero entries in this block can severely perturb the eigenvalues according to the Li--Li's bound \cref{eq:LiLi}, indicating that these entries should be eliminated first. In contrast, for the $\widetilde{A}_{12}$ and $\widetilde{A}_{21}$ blocks, the large spectral gap between the eigenvalues in $\widetilde{A}_{11}$ and $\widetilde{A}_{22}$ implies, again by the Li--Li's bound, that the nonzero entries there are essentially harmless for eigenvalue accuracy and thus should be deferred.
Although our strategy is greedy, its elimination order still conforms to the pattern suggested by the Li--Li's bound.
In contrast, the eigenvalue cluster in $\widetilde{A}_{11}$ causes the classical strategy to process this block last. This occurs because the off-diagonal entries in $\widetilde{A}_{11}$ are much smaller in magnitude than other off-diagonal entries of $\widetilde{A}$ when the cluster is tight. To see this, suppose that $\Lambda_{1}=\lambda I+\epsilon D$ with $\norm{D}=1$, then 
\begin{equation*}
    \widetilde{A}_{11} = V_{1}^{\Ttran}(\lambda I+\epsilon D)V_{1} = \lambda I+\epsilon V_{1}^{\Ttran}DV_{1},
\end{equation*}
where $V_{1}$ contains orthonormal columns. It is clear that the magnitude of these off-diagonal entries are at most $\epsilon$, which can be smaller than $\mpr_{\ell}$, the lower precision.
In our experiments with $\alpha =1$, the off-diagonal entries within the cluster have magnitudes on the order of $10^{-8}$, whereas those in the other parts of the matrix are typically on the order of $10^{-6}$.
Consequently, our newly proposed strategy eliminates nonzero entries in a more effective order, whereas the classical strategy expends effort on parts that have little impact.
This leads to a substantial acceleration in convergence.

\begin{figure}[htbp]
    \centering
    \includegraphics[width=\figsizeT]{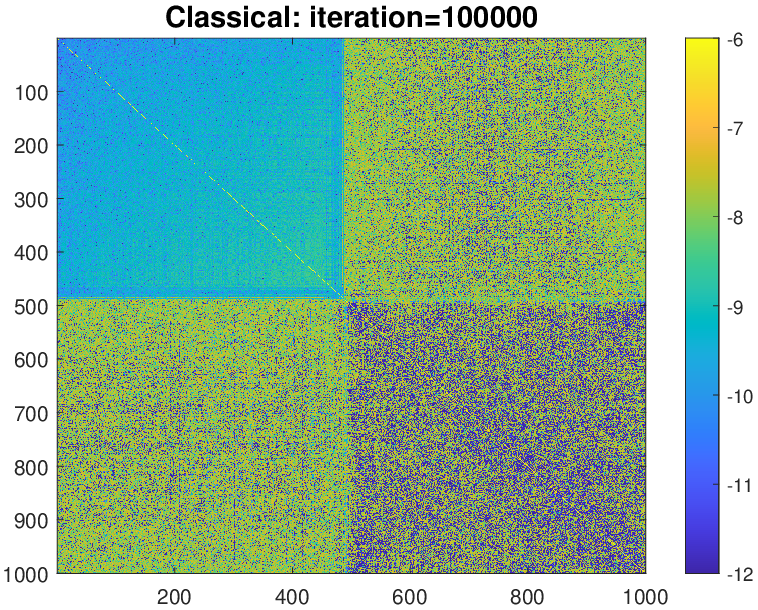}
    \includegraphics[width=\figsizeT]{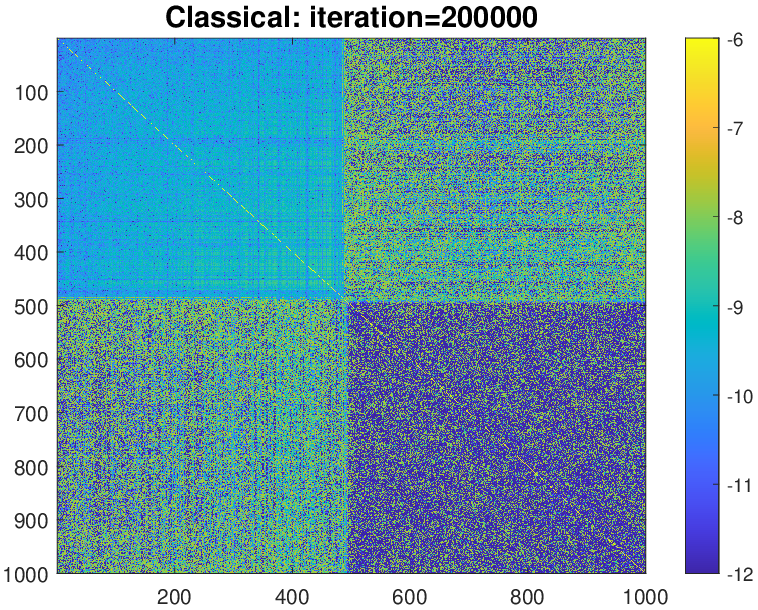}
    \includegraphics[width=\figsizeT]{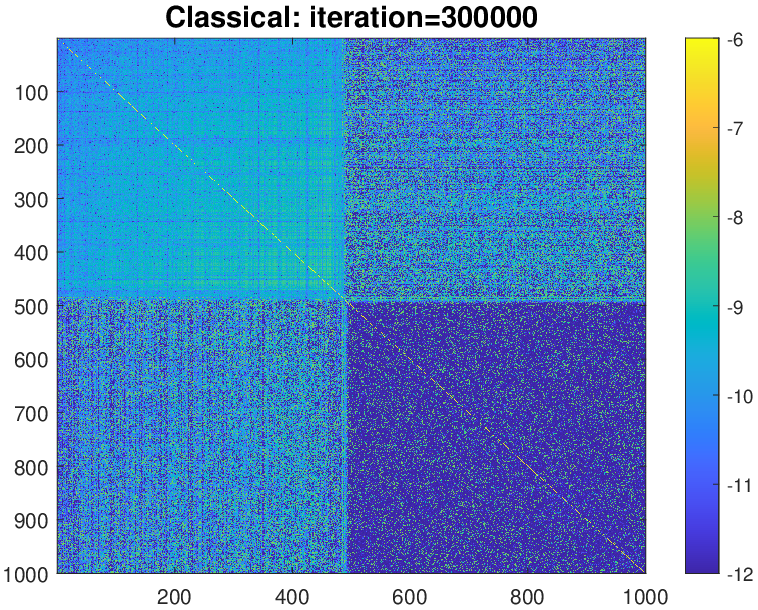}

    \includegraphics[width=\figsizeT]{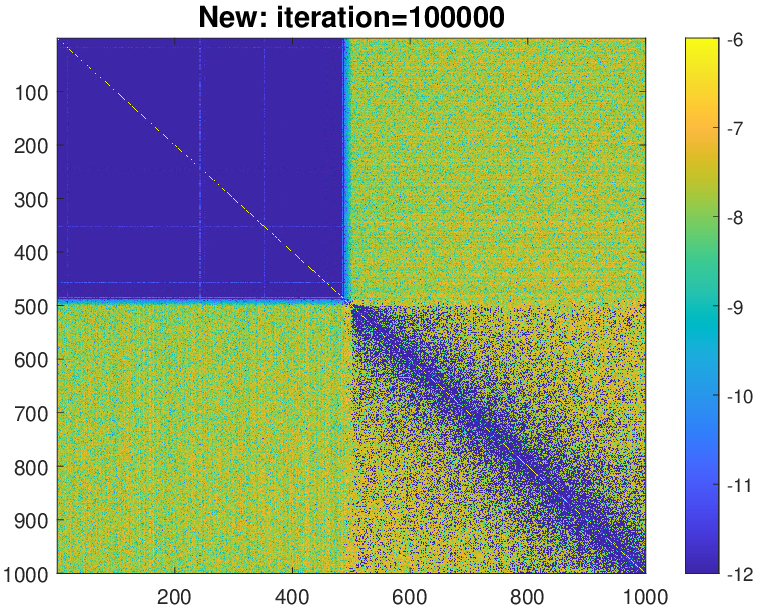}
    \includegraphics[width=\figsizeT]{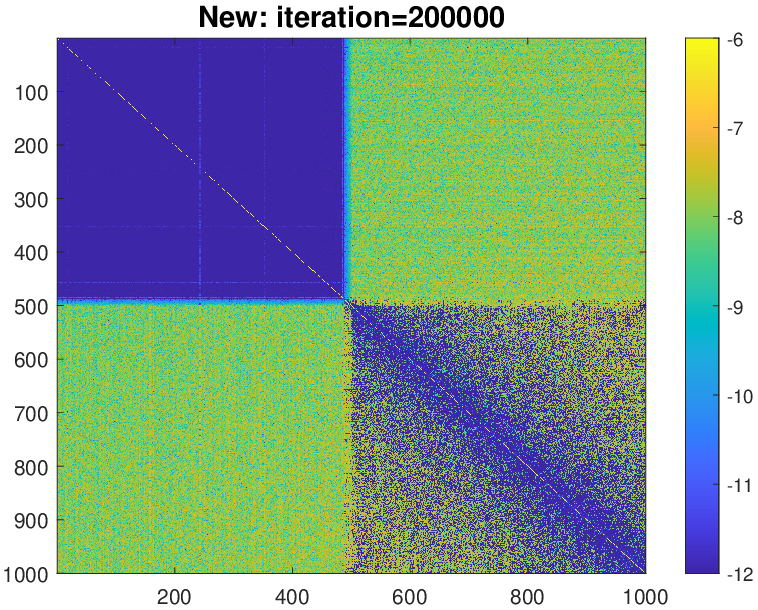}
    \includegraphics[width=\figsizeT]{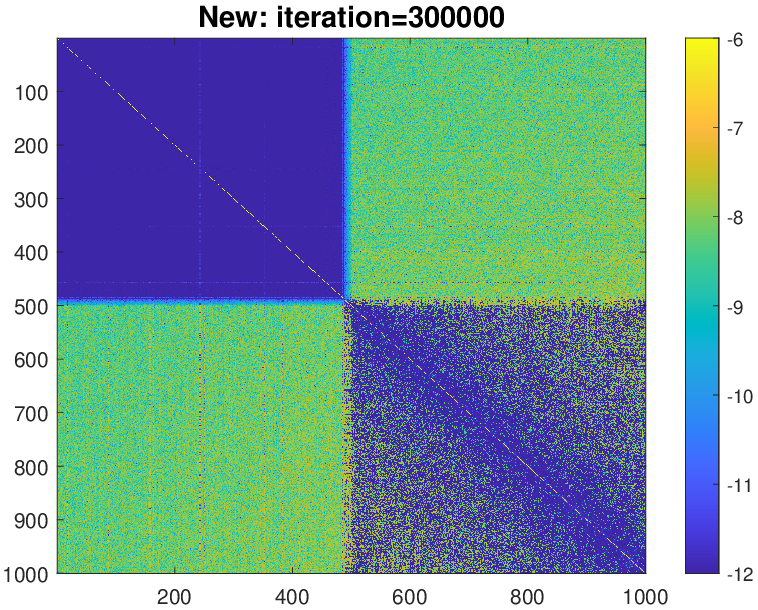}
    \caption{Magnitude of entries after few iterations. The upper-left block consists of a tight eigenvalue cluster.}
    \label{fig:hist}
\end{figure}

\subsection{Hilbert matrices}
Let $H_{n} = (h_{ij})\in\R^{n\times n}$ be the Hilbert matrix \cite[Sec.~28.1]{Higham2002}, where $h_{ij}=1/(i+j-1)$. It is a classical example of an ill-conditioned symmetric positive definite matrix, as it has exponentially decaying singular values \cite{Beckermann2019}. Consequently, computing the eigenvalues of $H_{n}$ with high \emph{relative} accuracy is a challenging task.
Although the Jacobi method can be more accurate than QR in some situations \cite{Demmel1992}, it cannot significantly improve relative accuracy here because the diagonal entries of $H_{n}$ differ by at most a factor of $n$. For the same reason, applying the one-sided Jacobi method to the \emph{computed} Cholesky factor is also ineffective, since the Cholesky factorization itself cannot be obtained with high relative accuracy.
However, as shown in \cite{Mathias1995}, applying the cyclic Jacobi method to the Cholesky factor $C_{n}$ given in \emph{closed form}, whose $(i,j)$ entry reads
\begin{equation*}
    c_{ij} = \frac{\sqrt{2j-1}\bigl((i-1)!\bigr)^{2}}{(i+j-1)!(i-j)!}\quad \text{for}\quad  i\geq j,
\end{equation*} 
yields eigenvalues with high relative accuracy.

Unfortunately, neither the classical greedy strategy nor our newly proposed strategy succeeds in this setting. This is because, for very small singular values (much smaller than $\mpr$), even an $\order(\mpr)$ off-diagonal perturbation could introduce a non-negligible relative error. 
Under absolute pivoting strategies, the off-diagonal entries near $\order(1)$ singular values are likely to be $\order(\mpr)$, so they are preferentially selected, while the smaller, but more harmful, off-diagonal entries associated with small singular values are never chosen.
To handle this situation, we
slightly modify our cost function into 
\begin{equation*}
    \widehat{L}_{ij}(A) = \Bigl(\frac{1}{a_{ii}}+\frac{1}{a_{jj}}\Bigr)L_{ij}(A),
\end{equation*}
where we assume that the matrix $A$ is symmetric positive definite.

Consider the $100\times 100$ Hilbert matrix $H_{100}$ and its Cholesky factor $C_{100}$. We first compute its singular values using MATLAB's symbolic computation as a reference solution. For our Jacobi method, we do \emph{not} employ mixed-precision preconditioning, since the associated (global) orthogonal transformations hinder the attainment of high relative accuracy. As illustrated in the left panel of \cref{fig:hilb}, our Jacobi method achieves higher relative accuracy than MATLAB's built-in \texttt{svd}. In the right panel, we compare our pivoting strategy with a randomized pivoting strategy, in which all pivot pairs $(i,j)$ are selected with equal probability.
Our strategy is clearly more efficient: after only 1000 pivots, the relative accuracy already surpasses that obtained by 50000 randomized pivots. Moreover, our method converges after approximately 100000 rotations, whereas the randomized approach fails to converge even after 200000 rotations.

\begin{figure}[htbp]
    \centering
    \includegraphics[width=\figsizeD]{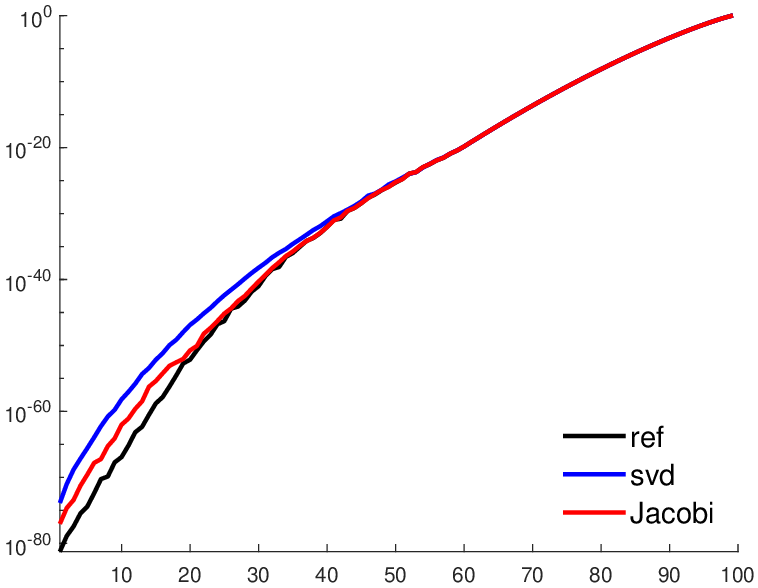}
    \includegraphics[width=\figsizeD]{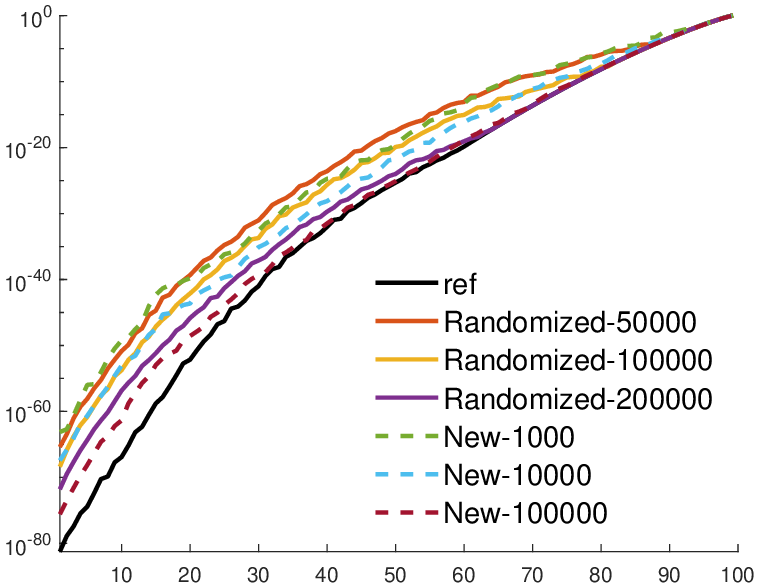}
    \caption{Singular values of $C_{100}$: reference (exact), and computed. The reference solution is computed via symbolic computation in MATLAB. For the right panel, the number following each label is the number of Jacobi rotations we have performed.}
    \label{fig:hilb}
\end{figure}

\section{Conclusion}
In this paper, we proposed a new pivoting strategy for accelerating the convergence of the Jacobi method by incorporating eigenvalue perturbation bounds. For the examples considered, our strategy significantly outperforms the classical greedy Jacobi method when there are some tight eigenvalue clusters. The strategy based on perturbation theory also has the potential to be applied to the block Jacobi method \cite{VanLoan1986} and achieve the communication lower bounds described in \cite{demmel2025minimizing}.

\bibliographystyle{abbrvurl}

\begin{thebibliography}{10}

\bibitem{Beckermann2019}
B.~Beckermann and A.~Townsend.
\newblock Bounds on the singular values of matrices with displacement structure.
\newblock {\em SIAM Rev.}, 61(2):319--344, 2019.
\newblock \href {https://doi.org/10.1137/19M1244433} {\path{doi:10.1137/19M1244433}}.

\bibitem{demmel2025minimizing}
J.~Demmel, H.~Luo, R.~Schneider, and Y.~Wang.
\newblock Minimizing the arithmetic and communication complexity of {J}acobi's method for eigenvalues and singular values.
\newblock {\em arXiv preprint arXiv:2506.03466}, 2025.
\newblock \href {https://doi.org/10.48550/arXiv.2506.03466} {\path{doi:10.48550/arXiv.2506.03466}}.

\bibitem{Demmel1992}
J.~Demmel and K.~Veseli\'c.
\newblock Jacobi's method is more accurate than {$QR$}.
\newblock {\em SIAM J. Matrix Anal. Appl.}, 13(4):1204--1245, 1992.
\newblock \href {https://doi.org/10.1137/0613074} {\path{doi:10.1137/0613074}}.

\bibitem{detherage2025unified}
I.~Detherage and R.~Shah.
\newblock A unified perspective on orthogonalization and diagonalization.
\newblock {\em arXiv preprint arXiv:2505.02023}, 2025.
\newblock \href {https://doi.org/10.48550/arXiv.2505.02023} {\path{doi:10.48550/arXiv.2505.02023}}.

\bibitem{Drmac2007}
Z.~Drma{\v c} and K.~Veseli\'c.
\newblock New fast and accurate {J}acobi {SVD} algorithm. {I}.
\newblock {\em SIAM J. Matrix Anal. Appl.}, 29(4):1322--1342, 2007.
\newblock \href {https://doi.org/10.1137/050639193} {\path{doi:10.1137/050639193}}.

\bibitem{Drmac2007a}
Z.~Drma{\v c} and K.~Veseli\'c.
\newblock New fast and accurate {J}acobi {SVD} algorithm. {II}.
\newblock {\em SIAM J. Matrix Anal. Appl.}, 29(4):1343--1362, 2007.
\newblock \href {https://doi.org/10.1137/05063920X} {\path{doi:10.1137/05063920X}}.

\bibitem{Gao2025}
W.~Gao, Y.~Ma, and M.~Shao.
\newblock A mixed precision {J}acobi {SVD} algorithm.
\newblock {\em ACM Trans. Math. Software}, 51(1):Art. 5, 33, 2025.

\bibitem{Hestenes1958}
M.~R. Hestenes.
\newblock Inversion of matrices by biorthogonalization and related results.
\newblock {\em J. Soc. Indust. Appl. Math.}, 6:51--90, 1958.

\bibitem{Higham2002}
N.~J. Higham.
\newblock {\em Accuracy and stability of numerical algorithms}.
\newblock Society for Industrial and Applied Mathematics (SIAM), Philadelphia, PA, second edition, 2002.
\newblock \href {https://doi.org/10.1137/1.9780898718027} {\path{doi:10.1137/1.9780898718027}}.

\bibitem{Higham2022}
N.~J. Higham and T.~Mary.
\newblock Mixed precision algorithms in numerical linear algebra.
\newblock {\em Acta Numer.}, 31:347--414, 2022.
\newblock \href {https://doi.org/10.1017/S0962492922000022} {\path{doi:10.1017/S0962492922000022}}.

\bibitem{Higham2025}
N.~J. Higham, F.~Tisseur, M.~Webb, and Z.~Zhou.
\newblock Computing {A}ccurate {E}igenvalues using a {M}ixed-{P}recision {J}acobi {A}lgorithm.
\newblock {\em SIAM J. Matrix Anal. Appl.}, 46(4):2423--2448, 2025.
\newblock \href {https://doi.org/10.1137/25M1723748} {\path{doi:10.1137/25M1723748}}.

\bibitem{Jacobi1846}
C.~G.~J. Jacobi.
\newblock \"uber ein leichtes {V}erfahren die in der {T}heorie der {S}\"acularst\"orungen vorkommenden {G}leichungen numerisch aufzul\"osen.
\newblock {\em J. Reine Angew. Math.}, 30:51--94, 1846.
\newblock \href {https://doi.org/10.1515/crll.1846.30.51} {\path{doi:10.1515/crll.1846.30.51}}.

\bibitem{Li2005}
C.-K. Li and R.-C. Li.
\newblock A note on eigenvalues of perturbed {H}ermitian matrices.
\newblock {\em Linear Algebra Appl.}, 395:183--190, 2005.
\newblock \href {https://doi.org/10.1016/j.laa.2004.08.026} {\path{doi:10.1016/j.laa.2004.08.026}}.

\bibitem{Mathias1995}
R.~Mathias.
\newblock Accurate eigensystem computations by {J}acobi methods.
\newblock {\em SIAM J. Matrix Anal. Appl.}, 16(3):977--1003, 1995.
\newblock \href {https://doi.org/10.1137/S089547989324820X} {\path{doi:10.1137/S089547989324820X}}.

\bibitem{VanLoan1986}
C.~F. Van~Loan.
\newblock The block {J}acobi method for computing the singular value decomposition.
\newblock In {\em Computational and combinatorial methods in systems theory ({S}tockholm, 1985)}, pages 245--255. North-Holland, Amsterdam, 1986.

\bibitem{Veselic1989}
K.~Veseli\'c and V.~Hari.
\newblock A note on a one-sided {J}acobi algorithm.
\newblock {\em Numer. Math.}, 56(6):627--633, 1989.
\newblock \href {https://doi.org/10.1007/BF01396349} {\path{doi:10.1007/BF01396349}}.

\bibitem{Wilkinson1962}
J.~H. Wilkinson.
\newblock Note on the quadratic convergence of the cyclic {J}acobi process.
\newblock {\em Numer. Math.}, 4:296--300, 1962.
\newblock \href {https://doi.org/10.1007/BF01386321} {\path{doi:10.1007/BF01386321}}.

\bibitem{zhou2026computing}
Z.~Zhou, F.~Tisseur, and M.~Webb.
\newblock Computing accurate singular values using a mixed-precision one-sided {J}acobi algorithm.
\newblock {\em arXiv preprint arXiv:2602.18134}, 2026.
\newblock \href {https://doi.org/10.48550/arXiv.2602.18134} {\path{doi:10.48550/arXiv.2602.18134}}.

\end{thebibliography}
\end{document}